\documentclass[11pt]{article}
\usepackage{amsfonts, amsmath, amssymb, amscd, amsthm, color, graphicx, mathrsfs, mathabx, wasysym, setspace, mdwlist, calc,float}
\usepackage{setspace}
\usepackage{hyperref}
\usepackage{tikz-cd} 
\usepackage{hyperref}
\usepackage{tocloft}
\usepackage{xcolor}

\usepackage{hyperref}
\hypersetup{linktocpage}

\hypersetup{colorlinks,
    linkcolor={red!50!black},
    citecolor={blue!80!black},
    urlcolor={blue!80!black}}
\usepackage{float}

\usepackage{comment}

\newcommand{\NN}{\mathbb{N}}
\newcommand{\ZZ}{\mathbb{Z}}

\newcommand{\NNo}{\mathbb{N}\cup\{0\}}

\newcommand{\inj}{\hookrightarrow}

\newcommand{\act}{\curvearrowright}

\newcommand{\G}{\mathcal{G}}
\newcommand{\R}{\mathfrak{R}}
\renewcommand{\H}{\mathcal{H}}

\newtheorem{thm}{Theorem}[section]
\newtheorem{cor}[thm]{Corollary}

\newtheorem{prop}[thm]{Proposition}

\theoremstyle{definition}
\newtheorem{defn}[thm]{Definition}

\theoremstyle{remark}
\newtheorem{rem}[thm]{Remark}

\newtheorem{que}[thm]{Question}

\newcommand{\la}{\langle}
\newcommand{\ra}{\rangle}

\newcommand{\stab}{\mathbf{Stab}}

\begin{document}

\title{Borel complexity for product of trees and commuting partial maps}

\author{Koichi Oyakawa}
\date{}

\maketitle

\vspace{-10mm}

\begin{abstract}
    We prove that every acylindrical action on uniformly locally finite product of trees induces the hyperfinite orbit equivalence relation on the Roller boundary. As a byproduct, we construct an example of a standard Borel space and two commuting bounded-to-1 surjective partial Borel maps that generate a universal countable Borel equivalence relation. This contrasts to Shinko-Weilacher-Yu's theorem on hyperfiniteness of bounded-to-one actions of commutative monoids.
\end{abstract}

\section{Introduction}
Borel complexity measures how complicated it is to describe equivalence relations on standard Borel spaces. It provides a formulation of classification problems, and the study of Borel complexity has long been an active topic in descriptive set theory. Recently, orbit equivalence relations appearing naturally in geometric group theory started to be explored partially motivated by a fundamental long standing open problem asking whether every measure-hyperfinite countable Borel equivalence relation is hyperfinite (see \cite[Section 16.4]{Kec25}). Because various groups acting on nonpositively curved spaces are known to induce topologically amenable actions on suitable boundaries and hence measure-hyperfinite orbit equivalence relations (see \cite{Ada94,Kai04, Oza06,Kid08,Ham09,Lec10,GN11,NS13,HH21,BGC22}), it is natural to ask whether these equivalence relations are hyperfinite. 

This question has been studied for group actions on hyperbolic spaces and their induced actions on the Gromov boundaries (see \cite{Mar19,HSS20,MS20,PS21,Kar22,Oya24,KEOSS24,NV25,KOO26}). \cite{Oya26} provided the first result for non-hyperbolic nonpositively curved spaces, where it was proved that every group acting virtually specially on a CAT(0) cube complex with finitely many hyperplane orbits induces the hyperfinite orbit equivalence relation on the Roller boundary. This result opened up the possibility to investigate Borel complexity of boundary actions for CAT(0) cube complexes. In fact, Guentner-Niblo showed in \cite{GN11} that if a countable group acts on a finite dimensional CAT(0) cube complex, then the induced action on the Roller compactification is topologically amenable if and only if every vertex stabilizer is amenable (see \cite[Section 4]{GN11} and \cite{BCGNW09,NS13}).

In this paper, we study Borel complexity of boundary actions for product of trees, which is another important class of CAT(0) cube complexes \cite{Oya26} could not cover. Theorem \ref{thm:product of trees} is also a generalization of \cite[Theorem A]{KEOSS24} in the uniformly locally finite case. See Definition \ref{def:group action} for $\stab_G(\gamma), \stab_G(p)$.

\begin{thm}\label{thm:product of trees}
    Let $G$ be a countable group and $T_1,\cdots,T_n$ be uniformly locally finite trees with $n \in \NN$. Let $G$ act on $X=T_1\times \cdots\times T_n$ cubically. If every geodesic ray $\gamma$ in $X$ has an initial segment $p$ such that $\stab_G(\gamma) = \stab_G(p)$, then the orbit equivalence relation $E_G^{\partial_\R X}$ induced by the action of $G$ on the Roller boundary $\partial_\R X$ of $X$ is hyperfinite.
\end{thm}

Importantly, Theorem \ref{thm:product of trees} applies to geometric actions on product of trees, which contains Burger-Mozes simple groups \cite{BM97,BM00}, Wise's CSC groups \cite{Wis07}, and other non residually finite CAT(0) cubical groups (see \cite{Rat07,JW09,BK22,LLM23,AL26}). The reason \cite{Oya26} couldn't cover actions of these groups is because they inherently lack nice coloring of hyperplanes that virtually special actions on CAT(0) cube complexes have. We overcome this difficulty by using product structure and applying a recent result \cite[Theorem 1.5]{SWY26} by Shinko-Weilacher-Yu, where they proved that every equivalence relation generated by countably many commuting bounded-to-one Borel maps is hyperfinite.

The Shinko-Weilacher-Yu's theorem solves the bounded-to-one case of a well-known folklore open problem asking whether every Borel action of a countable commutative monoid generates a hypersmooth equivalence relation (see \cite[Question 1.4]{SWY26}), which is motivated by Gao-Jackson's theorem solving Weiss's conjecture for abelian groups (see \cite{GJ15}). It turns out that this problem has a negative answer in a strong sense when we attempt extension to partial maps, which we prove in Theorem \ref{thm:commuting Einfty} below.

\begin{thm}\label{thm:commuting Einfty}
    There exists a standard Borel space $X$, a Borel subset $B \subset X$, and bounded-to-1 surjective Borel maps $f_1,f_2 \colon B \to X$ such that $f_1 f_2 =f_2f_1$ on $B \cap f_1^{-1}(B) \cap f_2^{-1}(B)$ and the equivalence relation $E_t^X(f_1,f_2)$ on $X$ generated by $f_1$ and $f_2$ is a universal countable Borel equivalence relation.
\end{thm}

Theorem \ref{thm:commuting Einfty} contrasts to both \cite[Theorem 1.5]{SWY26} on commutative monoid actions and Remark \ref{rem:single case} on the case of single functions.
 
This paper is organized as follows. In Section \ref{sec:Preliminaries}, we prepare necessary definitions and notations. In Section \ref{sec:Boundary actions for product of trees}, we prove Theorem \ref{thm:product of trees}. In Section \ref{sec:Commuting partial maps}, we prove Theorem \ref{thm:commuting Einfty}. In both Section \ref{sec:Boundary actions for product of trees} and Section \ref{sec:Commuting partial maps}, we present intermediate results to well explain how the study of boundary actions of product of trees inspired the construction in Theorem \ref{thm:commuting Einfty}.

\noindent\textbf{Acknowledgment.}
I would like to thank Clinton Conley, Jingyin Huang, Andrew Marks, Forte Shinko, and Felix Weilacher for helpful discussions. No AI was used for proving any results in this paper nor for writing this paper.

\section{Preliminaries}\label{sec:Preliminaries}
For $n \in \NN$, we define $[n]$ by $[n]=\{1,\cdots,n\}$.

\begin{defn}
    Let $X$ be a graph.  We denote by $X^{(0)}$ and $X^{(1)}$ the set of vertices and edges of $X$ respectively. The graph $X$ is called \textbf{uniformly locally finite} if there exists $M\in\NN$ such that the valency of every vertex is at most $M$. Let $p$ be a path in $X$, where we denote $p=(p(0),\cdots,p(n))$ with $n\in \NNo$. Note that $p(i)$ and $p(i+1)$ are adjacent vertices in $X$. Define $p_-, p_+, |p|$ by $p_-=p(0)$, $ p_+=p(n)$, and $|p|=n$. When $p$ has no self-intersection, for $a=p(i)$ and $b=p(j)$ with $0\le i\le j\le n$, we denote by $p_{[a,b]}$ the subpath of $p$ from $a$ to $b$ i.e. $p_{[a,b]}=(p(i),\cdots,p(j))$. A \textbf{ray} in $X$ is a graph homomorphism from $[0,\infty)$ to $X$. The notations $\gamma_-$ and $\gamma_{[a,b]}$ are defined similarly for rays. A path $p$ in $X$ from $x\in X^{(0)}$ to $y\in X^{(0)}$ is called \textbf{geodesic} if $|p|$ is the minimum among all paths from $x$ to $y$. A subpath $\gamma_{[\gamma_-,v]}$ of a geodesic ray $\gamma$ in $X$ with some vertex $v \in \gamma^{(0)}$ is called an \textbf{initial segment} of $\gamma$. A \textbf{tree} is a connected graph without any cycle.
\end{defn}

\begin{defn}\label{def:group action}
    Let a group $G$ act on a set $X$. For $a \in X$ and $A \subset X$, we define $\stab_G(a)$ and $\stab_G(A)$ by $\stab_G(a)=\{g \in G \mid ga=a\}$ and $\stab_G(A)=\bigcap_{a \in A}\stab_G(a)$. The action $G \act X$ is called \textbf{free} if $\stab_G(x)=\{1\}$ for any $x \in X$. We denote by $E_G^X$ the orbit equivalence relation of the action $G \act X$ i.e. for $x,y \in X$, $x\, E_G^X \,y \iff \exists\,g\in G,\, gx=y$. 
\end{defn}

\begin{defn}\label{def:concepts on CAT(0) cc}
    Let $X$ be a CAT(0) cube complex. For a subcomplex $C$ of $X$, we denote by $\H(C)$ the set of all hyperplanes of $X$ dual to some 1-cube of $C$. For $A,B \subset X^{(0)}$, we define $\H_X(A,B)$ to be the set of all hyperplanes that separate $A$ and $B$. (We denote $\H_X(a,B)$ when $A=\{a\}$ for brevity.) When $e$ is a 1-cube of $X$, we use $\H(e)$ to mean an element of $\H(X)$ as well as a singleton by abuse of notation. We denote by $\partial_\R X$ and $\G_X$ the Roller boundary of $X$ and the median graph on $\partial_\R X$ (see \cite[Section 2.1]{Oya25}).
\end{defn}

\begin{defn}\label{def:Borel eq rel}
    Let $X$ be a standard Borel space. A Borel equivalence relation $E$ on $X$ is called \textbf{hyperfinite} if there exist finite Borel equivalence relations $\{E_n\}_{n \in \NN}$ on $X$ such that $E = \bigcup_{n \in \NN}E_n$ and $\forall\,n \in \NN,\,E_n \subset E_{n+1}$. A countable Borel equivalence relation $E$ on $X$ is called \textbf{universal} if every countable Borel equivalence relation $F$ on a standard Borel space $Y$ Borel reduces to $E$ i.e. $F \le_B E$. For a set $(f_i)_{i \in I}$ of partial functions, where $f_i \colon B_i \to X$ with $B_i \subset X$ for each $i \in I$, we denote by $E_t^X((f_i)_{i \in I})$ the equivalence relation on $X$ generated by $(f_i)_{i \in I}$, that is, $E_t^X((f_i)_{i \in I})$ is the transitive closure of the relation $\sim$ on $X$ defined by $x \sim y$ if there exists $i \in I$ such that either $x\in B_i$ and $f_i(x)=y$, or $y\in B_i$ and $f_i(y)=x$ holds. A partial map $f \colon B \to X$ with $B \subset X$ is called \textbf{bounded-to-1} if there exists $M \in \NN$ such that $\sup_{x \in X}|f^{-1}(x)| \le M$.
\end{defn}

\section{Boundary actions for product of trees}\label{sec:Boundary actions for product of trees}

First, we prove Theorem \ref{thm:product of trees}. Note that \cite[Theorem A]{KEOSS24} is not applicable to each tree consisting of the product even in the case of diagonal actions. This is because even if a diagonal action satisfies the condition in Theorem \ref{thm:product of trees}, the induced action on a factor tree need not satisfy the condition of \cite[Theorem A]{KEOSS24}. The action in Proposition \ref{prop:Roller mod G is not hyperfin} is such an example.

Recall that $\stab_G(\gamma)$ and $\stab_G(p)$ below are the pointwise stabilizer of $\gamma$ and $p$ in $G$ respectively (see Definition \ref{def:group action}).

\begin{proof}[\textbf{Proof of Theorem \ref{thm:product of trees}}]
    We can naturally identify $\H(X)$ and $\bigsqcup_{i=1}^n T_i^{(1)}$. There exists a finite index subgroup $H$ of $G$ such that $H\cdot T_i^{(1)} = T_i^{(1)}$ for any $i \in [n]$. Indeed, $e \in T_i^{(1)}$ and $f \in T_j^{(1)}$ cross if and only if $i \neq j$, hence $G \act \H(X)$ induces the action $G\act [n]$ and we can take as $H$ the kernel of the action $G \act [n]$. By \cite[Proposition 1.3.(vii)]{JKL02}, it's enough to show that $E_H^{\partial_\R X}$ is hyperfinite. Fix $o \in X^{(0)}$. For $I \subset [n]$, define $A_I$ by
    \begin{align*}
        A_I &= \{x \in \partial_\R X \mid i \in I \iff |\H_X(o,x)\cap T_i^{(1)}|=\infty\}.
    \end{align*}
    Note $\partial_\R X = \bigsqcup_{I\subset [n]} A_I$ and each $A_I$ is $H$-invariant. To prove hyperfinitness of $E_H^{\partial_\R X}$, it's enough to prove that $E_H^{A_I}$ is hyperfinite for every $I\subset [n]$, which we will do in the following.

    Let $I\subset [n]$. Define $Z$ to be the set of isometric cubical maps $f \colon \NN^I \to X$ such that
    \begin{align}\label{eq:define Z}
        \H(f(\alpha_i)) \subset T_i^{(1)} 
        {\rm~for ~every~} i \in I,
        {\rm ~ where}~ \alpha_i=\{m\colon I\to\NN\mid \forall\,j \neq i,\, m(j)=1 \}.
    \end{align}
    Here, we consider $\NN^I$ as a CAT(0) cube complex with integer points being the set of 0-cubes. For $f \in Z$ and $k \in \NN$, set $x_{f,k} = f((k)_{i \in I}) \in X^{(0)}$. The group $H$ acts on $Z$ by $(g\cdot f) (x)= gf(x)$ for $g \in H$, $f \in Z$, and $x\in \NN^I$. 
    
    For each $i \in I$, define $\tau'_i \colon \NN^I \to \NN^I$ and $\tau_i \colon Z \to Z$ by 
    \begin{align*}
    (\tau_i'm)(j)=
    \begin{cases}
        m(j)+1 &{\rm if~}j=i\\
        m(j)   &{\rm if~}j\neq i~
    \end{cases}
    {\rm where}~m\colon I \to \NN,{\rm ~~~~~and~~~~~}
    \tau_i (f) = f \circ \tau_i'.   
    \end{align*}
    Note $\tau_i \circ \tau_j = \tau_j \circ \tau_i$ for any $i,j \in I$. We have $\tau_i \circ g = g \circ \tau_i$ for any $i \in I$ and $g \in G$. Since $X$ is uniformly locally finite and $X$ is nonpositively curved, there exists $K\in \NN$ such that $|\tau_i^{-1}(f)| \le K$ for any $f \in Z$ and $i \in I$.

    We show that $E_H^Z$ is smooth. For $R \in \NN$, define $\Delta_R$ to be the set of isometric cubical maps $f\colon [R]^I \to X$ such that $\H(f([R]^I \cap \alpha_i)) \subset T_i^{(1)}$ for every $i \in I$, where $\alpha_i$ is as in \eqref{eq:define Z}. Define $\Delta$ by $\Delta=\bigsqcup_{R \in \NN} \Delta_R$. For each $\psi \in \Delta_R$, define $Z_\psi$ to be set of all $f \in Z$ such that 
    \begin{align*}
        \psi(x)=f(x) ~{\rm for ~every}~ x \in [R]^I {\rm ~and~} \stab_H(x_{f,1}) \cap \stab_H(x_{f,R}) = \stab_H(f).
    \end{align*} 
    We claim $Z= \bigcup_{\psi \in \Delta} Z_\psi$. Indeed, let $f \in Z$. Fix a geodesic ray $\gamma$ in $X$ from $x_{f,1}$ that passes through $\{x_{f,k}\}_{k \in \NN}$. By our condition on $G \act X$, there exists an initial segment $p$ of $\gamma$ such that $\stab_G(\gamma) = \stab_G(p)$. Take $k \in \NN$ such that $p \subset \gamma_{[x_{f,1}, x_{f,k}]}$, then we have $\stab_H(\gamma) \subset \stab_H(\gamma_{[x_{f,1}, x_{f,k}]}) \subset \stab_H(p) = \stab_H(\gamma)$. By $\forall\,i \in [n],\,H\cdot T_i^{(1)} = T_i^{(1)}$, we have $\stab_H(x_{f,1}) \cap \stab_H(x_{f,k}) = \stab_H(\gamma_{[x_{f,1}, x_{f,k}]})$ and $\stab_H(f)=\stab_H(\gamma)$. Hence, $\stab_H(x_{f,1}) \cap \stab_H(x_{f,k}) = \stab_H(f)$. Define $\psi \in \Delta_k$ by $\psi=f|_{[k]^I}$, then we have $\psi \in Z_\psi$. Thus, the claim follows.
    
    For any $\psi \in \Delta$ and $f_1,f_2 \in Z_\psi$, we have $f_1\, E_H^Z|_{Z_\psi}\, f_2 \iff f_1=f_2$. Hence, $E_H^Z|_{Z_\psi}$ is smooth. Since $\Delta$ is countable, $E_H^Z$ is smooth by \cite[Lemma 3.1]{KEOSS24}.

    Since $E_H^Z$ is smooth, $Z/H$ is a standard Borel space. By $\forall\, g \in G,\, \forall\, i \in I,\, \tau_i \circ g = g \circ \tau_i$, the map $\widetilde \tau_i \colon Z/H \to Z/H$ is well-defined by $\widetilde \tau_i([f])=[\tau_i(f)]$ for each $i \in I$. By \cite[Theorem 1.5]{SWY26}, $E_t^{Z/H}((\widetilde \tau_i)_{i \in I})$ is hyperfinite. By \cite[Lemma 2.1]{KEOSS24}, we can see that the equivalence relation $F'$ on $Z$ generated by $\{\tau_i \mid  i \in I\}$ and $H$ is hyperfinite.

    Let $F_0'$ be the equivalence relation on $Z$ generated by $\{\tau_i \mid  i \in I\}$. Define the map $\varphi\colon Z \to A_I$ by $\varphi(f)=\lim_{k\ \to \infty}x_{f,k} \in A_I$. We can see $\varphi$ is surjective and for any $p,q \in Z$, we have $p \,F'_0\,q \iff \varphi(p)=\varphi(q)$. Hence, $Z/F_0'$ is Borel isomorphic to $A_I$. Since $\varphi$ is $H$-equivariant, $E_H^{A_I}$ is Borel bi-reducible to $E_H^{Z/F_0'}$. Thus, $E_H^{A_I}$ is hyperfinite by hyperfiniteness of $F'$ and \cite[Lemma 2.1]{KEOSS24}.
\end{proof}

\begin{rem}
    If \cite[Question 1.4]{SWY26} is solved positively, then Theorem \ref{thm:product of trees} is true without the condition that $T_i$'s are uniformly locally finite.
\end{rem}

An immediate corollary of Theorem \ref{thm:product of trees} is Corollary \ref{cor:product of trees} below. Importantly, Corollary \ref{cor:product of trees} applies to geometric actions on product of trees, which contains various groups \cite[Theorem 1.2]{Oya26} could not cover such as Burger-Mozes simple groups, Wise’s CSC groups, and other non residually finite CAT(0) cubical groups.

An action of a group $G$ on a metric space $(X,d_X)$ is called \textbf{acylindrical} if for any $\varepsilon > 0$, there exist $R,M > 0$ such that for any $x,y \in X$ with $d_X(x,y) \ge R$, we have
\begin{align*}
    \#\{g \in G\mid d_X(x,gx) \le \varepsilon {\rm ~and~} d_X(y,gy) \le \varepsilon \} \le M.
\end{align*}

\begin{cor}\label{cor:product of trees}
    Let $G$ be a countable group and $T_1,\cdots,T_n$ be uniformly locally finite trees. If $G$ acts on $X=T_1\times \cdots\times T_n$ acylindrically, then $E_G^{\partial_\R X}$ is hyperfinite.
\end{cor}

\begin{proof}
    Let $\gamma=(\gamma(0),\gamma(1),\cdots)$ be a geodesic ray in $X$. For $n \in \NN$, define $H_n$ by $H_n = \bigcap_{i=1}^n \stab_G(\gamma(i))$, then we have $H_n \supset H_{n+1}$ for any $n \in \NN$. Since $G \act X$ is acylindrical, there exists $N_0 \in \NN$ such that $H_{N_0}$ is finite. Hence, there exists $N_1 \ge N_0$ such that $H_n=H_{N_1}$ for any $n \ge N_1$. Thus, $\stab_G(\gamma) = H_{N_1} = \stab_G(\gamma_{[\gamma(0),\gamma(N_1)]})$. Hence, $E_G^{\partial_\R X}$ is hyperfinite by Theorem \ref{thm:product of trees}.
\end{proof}

Because hyperfiniteness for the Roller boundary in the case of virtually special groups actually followed from hyperfiniteness for the space of connected components of the Roller boundary (see \cite[Remark 3.13]{Oya26}), it is natural to ask whether the same is true in the case of product of trees. It turns out this is not the case. We prove it in Proposition \ref{prop:Roller mod G is not hyperfin} below.

In Proposition \ref{prop:Roller mod G is not hyperfin} below, $\partial_\R X/\G_X$ is the quotient of $\partial_\R X$ by the path-connectivity equivalence relation induced by $\G_X$. As a set, $\partial_\R X/\G_X$ is the set of path-connected components of the Borel median graph $\G_X$. By \cite[Theorem 1.1]{Oya25}, $\partial_\R X/\G_X$ is a standard Borel space.

\begin{prop}\label{prop:Roller mod G is not hyperfin}
    There exists a free action of a countable group $G$ on product of uniformly locally finite trees $X=T_1 \times T_2$ such that $E_G^{\partial_\R X/\G_X}$ is not measure-hyperfinite.
\end{prop}

\begin{proof}
Let $F_2$ be the free group of rank 2. Since the special linear group $SL_2(\ZZ)$ over $\ZZ$ decomposes as amalgamated free product $SL_2(\ZZ) = \ZZ/4\ZZ *_{\ZZ/2\ZZ}\ZZ/6\ZZ$, the group $SL_2(\ZZ)$ contains $F_2$ as a subgroup. 

Fix an embedding $\iota \colon F_2 \inj SL_2(\ZZ)$. For $n \in \NN$, let $q_n \colon SL_2(\ZZ) \to SL_2(\ZZ/2^n \ZZ)$ be the quotient homomorphism. Define $K_n \unlhd F_2$ by $K_n = \iota^{-1}(Ker(q_n))$. Set $K_0=F_2$ for convenience. The normal subgroups $\{K_n\}_{n = 0}^\infty$ are decreasing and satisfy $[K_{n-1}: K_n] \le 2^4$ for any $n \in \NN$. 

Construct a tree $T$ as follows. The vertex set $T^{(0)}$ of $T$ is $\bigsqcup_{n = 0}^\infty F_2/K_n$. For each $n \in \NN$, connect two vertices $aK_{n-1}$ and $bK_n$ with $x,y \in F_2$, if $xK_{n-1} \supset yK_n$. By $\forall n \in \NN,\,[K_{n-1}: K_n] \le 2^4$, $T$ is a uniformly locally finite tree. The action of $F_2$ on $T$ is defined by $g\cdot xK_n = gxK_n$ for $g,x \in F_2$ and $n \ge 0$. The relation $E_{F_2}^{\partial_\R T}$ is not measure-finite since $F_2 \act \partial_\R T$ is a free probability measure preserving action by \cite[Fact 1.4]{Kec05}.

Let $F_2=\la a,b \ra$ and let $\Gamma$ be the Cayley graph of $F_2$ with respect to $\{a,b\}^{\pm 1}$. Define the action $F_2 \act \Gamma \times T$ by $g\cdot (x,y) = (gx,gy)$. The action $F_2 \act \Gamma \times T$ is free since $F_2 \act \Gamma$ is free.

Set $X=\Gamma \times T$. Note $\partial_\R X = (\partial_\R \Gamma \times \partial_\R T) \sqcup (\partial_\R \Gamma \times T) \sqcup (\Gamma \times \partial_\R T)$. We have a $F_2$-equivariant Borel isomorphism $(\Gamma \times \partial_\R T)/\G_X \to \partial_\R T$. Hence, $E_{F_2}^{(\Gamma \times \partial_\R T)/\G_X}$ is not measure-hyperfinite. By this and $(\Gamma \times \partial_\R T)/\G_X \subset \partial_\R X / \G_X$, $E_{F_2}^{\partial_\R X / \G_X}$ is not measure-hyperfinite.
\end{proof}

While the action in Proposition \ref{prop:Roller mod G is not hyperfin} is proper, it is not cocompact. This raises the following question.

\begin{que}\label{que:geometric action}
    Suppose that a countable group $G$ acts on product of trees $X$ properly and cocompactly. Is $E_G^{\partial_\R X / \G_X}$ hyperfinite?
\end{que}

\begin{rem}
In Proposition \ref{prop:Roller mod G is not hyperfin}, we can take a diagonal action $G \act T_1\times T_2$ so that the action $G \act T_i$ is transitive for every $i \in \{1,2\}$ as follows. Take the action $F_2 \act T$ as in the proof of Proposition \ref{prop:Roller mod G is not hyperfin} such that $E_{F_2}^{\partial_\R T}$ is not measure-hyperfinite. Note that the valency of every vertex of $T$ is at most $2^4$. For each $n \in \NN$, let $T_n$ be the rooted tree with the root $o \in T_n^{(0)}$ such that the valency of $o$ is $n$ and the valency of all other vertices is $2^4$. For each $v \in T^{(0)}$ with valency $\ell$ satisfying $\ell < 16$, attach a copy $T_v$ of $T_{16-\ell}$ (i.e. there is an isomorphism $\iota_v \colon T_v \to T_{16-\ell}$) to $T$ by identifying $o$ and $v$. The resulting graph $T'$ is a 16-regular tree. Fix an action $F_8\act T'$, which is possible since $T'$ is isomorphic to the standard Cayley graph of $F_8$. Note that $T$ is a subtree of $T'$. Extend the action $F_2 \act T$ to $F_2 \act T'$ by defining the isomorphism $g \colon T_v \to T_{gv}$ for each $v \in T^{(0)}$ and $g \in F_2$ such that $\iota_{gv} \circ g = \iota_v$. Define the action $F_{10} \act T'$ by extending the homomorphisms $F_2\to Aut(T')$ and $F_8\to Aut(T')$ to $F_{10} \,(\cong F_2*F_8)$ as free product. By $E_{F_2}^{\partial_\R T} \subset E_{F_2}^{\partial_\R T'} \subset E_{F_{10}}^{\partial_\R T'}$, the relation $E_{F_{10}}^{\partial_\R T'}$ is not measure-hyperfinite. Let $F_{10}=\la a_1,\cdots,a_{10} \ra$ and let $\Gamma$ be the Cayley graph of $F_{10}$ with respect to $\{a_1,\cdots,a_{10}\}^{\pm 1}$. Set $G=F_{10}$, $T_1=T'$, and $T_2=\Gamma$, then the action $G \act T_i$ is transitive for each $i \in \{1,2\}$ as well as the product action $G \act T_1\times T_2$ satisfies the conditions in Proposition \ref{prop:Roller mod G is not hyperfin}. Note that this example still dose not answer Question \ref{que:geometric action} negatively.
\end{rem}

\section{Commuting partial maps}\label{sec:Commuting partial maps}

First, we turn the construction in Proposition \ref{prop:Roller mod G is not hyperfin} to that of commuting partial maps. The idea is creating the path-connectivity equivalence relation on a graph by shrinking geodesic paths.

\begin{prop}\label{prop:commuting non hyperfinite}
    There exists a standard Borel space $X$, a Borel subset $B \subset X$, and bounded-to-1 surjective Borel maps $f_1,f_2 \colon B \to X$ such that $f_1 f_2 =f_2f_1$ on $B \cap f_1^{-1}(B) \cap f_2^{-1}(B)$ and $E_t^X(f_1,f_2)$ is not measure-hyperfinite.
\end{prop}

\begin{proof}
    Let $F_2,\Gamma, T$ as in the proof of Proposition \ref{prop:Roller mod G is not hyperfin}. Define $P$ to be the set of all finite geodesic paths in $\Gamma$. Note that the vertex set $\Gamma^{(0)} \,(=F_2)$ of $\Gamma$ is contained in $P$ as geodesic paths of length $0$. The set $P$ is countable. The group $F_2$ acts on $P$ freely since $F_2 \act \Gamma$ is free. Define the action $F_2 \act P \times \partial_\R T$ by $g \cdot (p,\xi) = (gp,g\xi)$.

    Define $Q \subset P$ by $Q=\{p \in P \mid |p| \ge 1\}$. Define $\sigma_1, \sigma_2 \colon Q \to P$ as follow: for $p=(p_0,\cdots,p_n) \in P$ with $n \in \NN$ and $p_i \in \Gamma^{(0)}$,
    \begin{align}\label{eq:shrink}
        \sigma_1(p) = (p_1,\cdots,p_n)
        {\rm~~~and~~~}
        \sigma_2(p) = (p_0,\cdots,p_{n-1}).
    \end{align}
    Note $Q \cap \sigma_1^{-1}(Q) \cap \sigma_2^{-1}(Q) = \{p \in P \mid |p| \ge 2\}$. For any $p \in P$ with $|p| \ge 2$, we have 
    \begin{align}\label{eq:sigma 1 2 commute}
    \sigma_1\sigma_2 = \sigma_1\sigma_2.    
    \end{align}
    Define maps $\widetilde\sigma_1, \widetilde\sigma_2 \colon Q\times  \partial_\R T \to P \times  \partial_\R T$ by $\widetilde\sigma_i(p,\xi) = (\sigma_i(p),\xi)$ for each $i \in \{1,2\}$. Define $R$ by
    \begin{align*}
        R = E_t^{P \times \partial_\R T}(\widetilde\sigma_1, \widetilde\sigma_2).
    \end{align*}
    For any $v,w \in \Gamma^{(0)}$ and $\xi \in \partial_\R T$, let $p \in P$ be a geodesic in $\Gamma$ from $v$ to $w$, then we have $(v,\xi) \,R\, (w,\xi)$ by $\sigma_1^{|p|}(p)=w$ and $\sigma_2^{|p|}(p)=v$. Hence, the Borel map $\psi \colon P \times \partial_\R T \to \partial_\R T$ defined by $\psi(p,\xi) = \xi$ induces a bijection $(P \times \partial_\R T)/R \to \partial_\R T$. Let $E_{F_2}^{P \times \partial_\R T}\times R$ be the equivalence relation on $P \times \partial_\R T$ generated by $E_{F_2}^{P \times \partial_\R T}$ and $R$. Then, $E_{F_2}^{P \times \partial_\R T}\times R$ and $E_{F_2}^{\partial_\R T}$ are Borel bi-reducible. 
    
    For any $g \in F_2$, we have
    \begin{align}\label{eq:sigma and g commute}
        g\sigma_1=\sigma_1 g
        {\rm ~~~and~~~}
        g\sigma_2 =\sigma_2g.
    \end{align}
    For any $p \in P$ and $\xi,\eta \in \partial_\R T$, we have $(p,\xi) \, E_{F_2}^{P \times \partial_\R T}\, (p,\eta) \iff \xi = \eta$ since $F_2\act P$ is free. Hence, $E_{F_2}^{P \times \partial_\R T}$ is smooth on $\{p\} \times \partial_\R T$ for any $p \in P$. Since $P$ is countable, $E_{F_2}^{P \times \partial_\R T}$ is smooth. Hence, the quotient $(P \times \partial_\R T)/F_2$ is a standard Borel space. Set $X=(P \times \partial_\R T)/F_2$ and $B = (Q \times \partial_\R T)/F_2$. The set $B$ is Borel.

    Define the maps $f_1,f_2 \colon B \to X$ by $f_i([x])=[\widetilde\sigma_i(x)]$ for each $i \in \{1,2\}$, where $x \in Q \times \partial_\R T$ and $[x] \in B$ denotes the $E_{F_2}^{P \times \partial_\R T}$-class of $x$. By \eqref{eq:sigma and g commute}, $f_1, f_2$ are well-defined. Since $\Gamma$ is a 4-regular tree, we have $|f_i^{-1}(x)| \le  4$ for any $x \in X$ and $i \in \{1,2\}$. By \eqref{eq:sigma 1 2 commute}, we have $f_1 f_2 =f_2f_1$ on $B \cap f_1^{-1}(B) \cap f_2^{-1}(B)$.

     We can see that $E_t^X(f_1,f_2)$ and $E_{F_2}^{P \times \partial_\R T}\times R$ are Borel bi-reducible. Hence, $E_t^X(f_1,f_2)$ and $E_{F_2}^{\partial_\R T}$ are Borel bi-reducible. Since $E_{F_2}^{\partial_\R T}$ is not measure-hyperfinite, $E_t^X(f_1,f_2)$ is not measure-hyperfinite.
\end{proof}

\begin{rem}\label{rem:single case}
    Let $X$ be a standard Borel space, $B \subset X$ a Borel subset, and $f \colon B \to X$ a countable-to-1 Borel map (i.e. $f^{-1}(x)$ is countable for any $x \in X$). Then, the equivalence relation $E_t(f)$ generated by $f$ is hyperfinite. This can be shown as follows. Define a map $\widetilde f \colon X\to X$ by $\widetilde f(x) = f(x)$ on $B$ and $\widetilde f(x) = x$ on $X \setminus B$. The map $\widetilde f$ is countable-to-1 and Borel. We have $E_t(f) = E_t({\widetilde f})$. By \cite[Corollary 8.2]{DJK94}, $E_t({\widetilde f})$ is hyperfinite.
\end{rem}

In the same way as Proposition \ref{prop:commuting non hyperfinite}, we get a proof of Theorem \ref{thm:commuting Einfty} below.

\begin{proof}[\textbf{Proof of Theorem \ref{thm:commuting Einfty}}]
    Let $G$ be a finitely generated group acting on a standard Borel space $Y$ such that $E_G^Y$ is a universal countable Borel equivalence relation. Such $G, Y$ exist by \cite[Proposition 1.3]{DJK94}. Let $S$ be a finite generating set of $G$ and $\Gamma$ be the Cayley graph of $G$ with respect to $S$. Define $P$ to be the set of all finite paths in $\Gamma$ and define $Q \subset P$ by $Q=\{p\in P\mid |p|\ge1\}$. Define $\sigma_1,\sigma_2\colon Q\to P$ by $ \sigma_1(p) = (p_1,\cdots,p_n)$ and $\sigma_2(p) = (p_0,\cdots,p_{n-1})$ for $p=(p_0,\cdots,p_n)$ as in \eqref{eq:shrink}. Since $E_G^{P\times Y}$ is smooth, the quotient $X=(P\times Y)/G$ is a standard Borel space. Define $B\subset X$ and $f_1, f_2\colon B \to X$ by $B=(Q\times Y)/G$, and $f_i([(p, y)]) = [(\sigma_i(p), y)]$ for $(p,y) \in Q\times Y$, then $E_t^X(f_1,f_2)$ is a universal countable Borel equivalence relation.
\end{proof}


\begin{thebibliography}{KEOSS24}

\bibitem[Ada94]{Ada94}
S.~Adams, \emph{Boundary amenability for word hyperbolic groups and an application to smooth dynamics of simple groups}, Topology \textbf{33} (1994), no.~4, 765--783. \MR{1293309}

\bibitem[AL26]{AL26}
Michal Amir and Nir Lazarovich, \emph{Simple lattices in products of davis complexes}, 2026.

\bibitem[BCG{\etalchar{+}}09]{BCGNW09}
J.~Brodzki, S.~J. Campbell, E.~Guentner, G.~A. Niblo, and N.~J. Wright, \emph{Property {A} and {$\rm CAT(0)$} cube complexes}, J. Funct. Anal. \textbf{256} (2009), no.~5, 1408--1431. \MR{2490224}

\bibitem[BGH22]{BGC22}
Mladen Bestvina, Vincent Guirardel, and Camille Horbez, \emph{Boundary amenability of {${\rm Out}(F_N)$}}, Ann. Sci. \'Ec. Norm. Sup\'er. (4) \textbf{55} (2022), no.~5, 1379--1431. \MR{4517690}

\bibitem[BK22]{BK22}
Ievgen Bondarenko and Bohdan Kivva, \emph{Automaton groups and complete square complexes}, Groups Geom. Dyn. \textbf{16} (2022), no.~1, 305--332. \MR{4424972}

\bibitem[BM97]{BM97}
Marc Burger and Shahar Mozes, \emph{Finitely presented simple groups and products of trees}, C. R. Acad. Sci. Paris S\'er. I Math. \textbf{324} (1997), no.~7, 747--752. \MR{1446574}

\bibitem[BM00]{BM00}
\bysame, \emph{Lattices in product of trees}, Inst. Hautes \'Etudes Sci. Publ. Math. (2000), no.~92, 151--194. \MR{1839489}

\bibitem[DJK94]{DJK94}
R.~Dougherty, S.~Jackson, and A.~S. Kechris, \emph{The structure of hyperfinite {B}orel equivalence relations}, Trans. Amer. Math. Soc. \textbf{341} (1994), no.~1, 193--225.

\bibitem[GJ15]{GJ15}
Su~Gao and Steve Jackson, \emph{Countable abelian group actions and hyperfinite equivalence relations}, Invent. Math. \textbf{201} (2015), no.~1, 309--383. \MR{3359054}

\bibitem[GN11]{GN11}
Erik Guentner and Graham~A. Niblo, \emph{Complexes and exactness of certain {A}rtin groups}, Algebr. Geom. Topol. \textbf{11} (2011), no.~3, 1471--1495. \MR{2821432}

\bibitem[Ham09]{Ham09}
Ursula Hamenst\"adt, \emph{Geometry of the mapping class groups. {I}. {B}oundary amenability}, Invent. Math. \textbf{175} (2009), no.~3, 545--609. \MR{2471596}

\bibitem[HH21]{HH21}
Camille Horbez and Jingyin Huang, \emph{Boundary amenability and measure equivalence rigidity among two-dimensional artin groups of hyperbolic type}, 2021.

\bibitem[HSS20]{HSS20}
Jingyin Huang, Marcin Sabok, and Forte Shinko, \emph{Hyperfiniteness of boundary actions of cubulated hyperbolic groups}, Ergodic Theory Dynam. Systems \textbf{40} (2020), no.~9, 2453--2466. \MR{4130811}

\bibitem[JKL02]{JKL02}
S.~Jackson, A.~S. Kechris, and A.~Louveau, \emph{Countable {B}orel equivalence relations}, J. Math. Log. \textbf{2} (2002), no.~1, 1--80.

\bibitem[JW09]{JW09}
David Janzen and Daniel~T. Wise, \emph{A smallest irreducible lattice in the product of trees}, Algebr. Geom. Topol. \textbf{9} (2009), no.~4, 2191--2201. \MR{2558308}

\bibitem[Kai04]{Kai04}
Vadim~A. Kaimanovich, \emph{Boundary amenability of hyperbolic spaces}, Discrete geometric analysis, Contemp. Math., vol. 347, Amer. Math. Soc., Providence, RI, 2004, pp.~83--111. \MR{2077032}

\bibitem[Kar22]{Kar22}
Chris Karpinski, \emph{Hyperfiniteness of boundary actions of relatively hyperbolic groups}, 2022, To appear in Groups Geom. Dyn.

\bibitem[Kec05]{Kec05}
A.~S. Kechris, \emph{Unitary representations and modular actions}, Zap. Nauchn. Sem. S.-Peterburg. Otdel. Mat. Inst. Steklov. (POMI) \textbf{326} (2005), 97--144, 281--282. \MR{2183218}

\bibitem[Kec25]{Kec25}
Alexander~S. Kechris, \emph{The theory of countable {B}orel equivalence relations}, Cambridge Tracts in Mathematics, vol. 234, Cambridge University Press, Cambridge, 2025. \MR{4837611}

\bibitem[KEOSS24]{KEOSS24}
Srivatsav Kunnawalkam~Elayavalli, Koichi Oyakawa, Forte Shinko, and Pieter Spaas, \emph{Hyperfiniteness for group actions on trees}, Proc. Amer. Math. Soc. \textbf{152} (2024), no.~9, 3657--3664. \MR{4781963}

\bibitem[Kid08]{Kid08}
Yoshikata Kida, \emph{The mapping class group from the viewpoint of measure equivalence theory}, Mem. Amer. Math. Soc. \textbf{196} (2008), no.~916, viii+190. \MR{2458794}

\bibitem[KOO26]{KOO26}
Chris Karpinski, Damian Osajda, and Koichi Oyakawa, \emph{Graphical small cancellation and hyperfiniteness of boundary actions}, J. Lond. Math. Soc. (2) \textbf{113} (2026), no.~4, Paper No. e70516, 23. \MR{5055885}

\bibitem[L\'10]{Lec10}
Jean L\'ecureux, \emph{Amenability of actions on the boundary of a building}, Int. Math. Res. Not. IMRN (2010), no.~17, 3265--3302. \MR{2680274}

\bibitem[LLM23]{LLM23}
Nir Lazarovich, Ivan Levcovitz, and Alex Margolis, \emph{Counting lattices in products of trees}, Comment. Math. Helv. \textbf{98} (2023), no.~3, 597--630. \MR{4668544}

\bibitem[Mar19]{Mar19}
Timoth\'ee Marquis, \emph{On geodesic ray bundles in buildings}, Geom. Dedicata \textbf{202} (2019), 27--43. \MR{4001806}

\bibitem[MS20]{MS20}
Timoth\'{e}e Marquis and Marcin Sabok, \emph{Hyperfiniteness of boundary actions of hyperbolic groups}, Math. Ann. \textbf{377} (2020), no.~3-4, 1129--1153. \MR{4126891}

\bibitem[NS13]{NS13}
Amos Nevo and Michah Sageev, \emph{The {P}oisson boundary of {${\rm CAT}(0)$} cube complex groups}, Groups Geom. Dyn. \textbf{7} (2013), no.~3, 653--695. \MR{3095714}

\bibitem[NV25]{NV25}
Petr Naryshkin and Andrea Vaccaro, \emph{Hyperfiniteness and {B}orel asymptotic dimension of boundary actions of hyperbolic groups}, Math. Ann. \textbf{392} (2025), no.~1, 197--208. \MR{4887757}

\bibitem[Oya24]{Oya24}
Koichi Oyakawa, \emph{Hyperfiniteness of boundary actions of acylindrically hyperbolic groups}, Forum Math. Sigma \textbf{12} (2024), Paper No. e32, 31. \MR{4715159}

\bibitem[Oya25]{Oya25}
Koichi Oyakawa, \emph{Borel asymptotic dimension of the roller boundary of finite dimensional cat(0) cube complexes}, 2025, Preprint, arXiv:2505.10334.

\bibitem[Oya26]{Oya26}
\bysame, \emph{Hyperfiniteness of the boundary action of virtually special groups}, 2026.

\bibitem[Oza06]{Oza06}
Narutaka Ozawa, \emph{Boundary amenability of relatively hyperbolic groups}, Topology Appl. \textbf{153} (2006), no.~14, 2624--2630. \MR{2243738}

\bibitem[PS21]{PS21}
Piotr Przytycki and Marcin Sabok, \emph{Unicorn paths and hyperfiniteness for the mapping class group}, Forum Math. Sigma \textbf{9} (2021), Paper No. e36, 10. \MR{4252215}

\bibitem[Rat07]{Rat07}
Diego Rattaggi, \emph{A finitely presented torsion-free simple group}, J. Group Theory \textbf{10} (2007), no.~3, 363--371. \MR{2320973}

\bibitem[SWY26]{SWY26}
Forte Shinko, Felix Weilacher, and Jing Yu, \emph{Hyperfiniteness of bounded-to-one actions of commutative monoids}, 2026.

\bibitem[Wis07]{Wis07}
Daniel~T. Wise, \emph{Complete square complexes}, Comment. Math. Helv. \textbf{82} (2007), no.~4, 683--724. \MR{2341837}

\end{thebibliography}
\newcommand{\etalchar}[1]{$^{#1}$}
\providecommand{\bysame}{\leavevmode\hbox to3em{\hrulefill}\thinspace}
\providecommand{\MR}{\relax\ifhmode\unskip\space\fi MR }
\providecommand{\MRhref}[2]{%
  \href{http://www.ams.org/mathscinet-getitem?mr=#1}{#2}
}
\providecommand{\href}[2]{#2}

\vspace{5mm}

\noindent 970 Evans Hall, Department of Mathematics, University of California, Berkeley, CA 94720, USA.

\noindent E-mail: \emph{koichi.oyakawa@berkeley.edu}

\end{document}